\documentclass[11pt,leqno]{article}
\usepackage{amsthm,amsfonts,amssymb,amsmath,oldgerm}
\usepackage{epsfig}
\numberwithin{equation}{section}
\usepackage[thinlines]{easybmat}

\newcommand{\bu}{\bar{u}}

\newcommand{\txi}{\tilde{\xi}}
\newcommand{\x}{{\xi}_1} 

\newcommand{\td}{\tilde}

\newcommand{\hv}{\hat{v}}
\newcommand{\curl}{\hbox{ \rm curl }}

\def\txi{{\tilde \xi}}

\def\eps{\varepsilon }

\newcommand\R{\mathbb R}

\def\eps{\varepsilon}

\newcommand\kernel{\hbox{\rm Ker}}

\newcommand\br{\begin{remark}}
\newcommand\er{\end{remark}}
\newcommand\bp{\begin{pmatrix}}
\newcommand\ep{\end{pmatrix}}
\newcommand\be{\begin{equation}}
\newcommand\ee{\end{equation}}
\newcommand\ba{\begin{equation}\begin{aligned}}
\newcommand\ea{\end{aligned}\end{equation}}

\newcommand{\bap}{\begin{app}}
\newcommand{\eap}{\end{app}}
\newcommand{\begs}{\begin{exams}}
\newcommand{\eegs}{\end{exams}}
\newcommand{\beg}{\begin{example}}
\newcommand{\eeg}{\end{exaplem}}
\newcommand{\bpr}{\begin{proposition}}
\newcommand{\epr}{\end{proposition}}
\newcommand{\bt}{\begin{theorem}}
\newcommand{\et}{\end{theorem}}
\newcommand{\bc}{\begin{corollary}}
\newcommand{\ec}{\end{corollary}}
\newcommand{\bl}{\begin{lemma}}
\newcommand{\el}{\end{lemma}}
\newcommand{\bd}{\begin{definition}}
\newcommand{\ed}{\end{definition}}
\newcommand{\brs}{\begin{remarks}}
\newcommand{\ers}{\end{remarks}}

\newtheorem{theo}{Theorem}[section]
\newtheorem{prop}[theo]{Proposition}
\newtheorem{cor}[theo]{Corollary}
\newtheorem{lem}[theo]{Lemma}

\newtheorem{exams}[theo]{Examples}

\numberwithin{equation}{section}

\newcommand{\RR}{{\mathbb R}}

\newcommand{\const}{\text{\rm constant}}

\newtheorem{theorem}{Theorem}[section]
\newtheorem{proposition}[theorem]{Proposition}
\newtheorem{corollary}[theorem]{Corollary}
\newtheorem{lemma}[theorem]{Lemma}
\newtheorem{definition}[theorem]{Definition}

\newtheorem{example}[theorem]{Example}
\newtheorem{remark}[theorem]{Remark}

\newcommand{\RM}{\mathbb{R}}

\title{
Nonlinear stability of periodic traveling wave solutions of
viscous conservation laws in
dimensions one and two}

\author{\sc \small
Mathew A. Johnson\thanks{Indiana University, Bloomington, IN 47405;
matjohn@indiana.edu: Research of M.J. was partially supported by an NSF Postdoctoral Fellowship under NSF grant DMS-0902192.}
~~~~~
Kevin Zumbrun\thanks{Indiana University, Bloomington, IN 47405;
kzumbrun@indiana.edu:
Research of K.Z. was partially supported
under NSF grants no. DMS-0300487 and DMS-0801745.
 }}
\begin{document}

\maketitle


\begin{center}
\end{center}

\begin{center}
{\bf 2000 MR Subject Classification}: 35B35.
\end{center}


\begin{abstract}
Extending results of Oh and Zumbrun in dimensions $d\ge 3$,
we establish nonlinear stability and asymptotic behavior
of spatially-periodic traveling-wave solutions of viscous systems of
conservation laws in critical
dimensions $d=1,2$, under a natural set of spectral stability assumptions
introduced by Schneider in the setting of reaction diffusion equations.
The key new steps in the analysis beyond that in dimensions $d\ge 3$
are a refined Green function estimate separating off translation as
the slowest decaying linear mode and a novel scheme for detecting
cancellation at the level of the nonlinear iteration
in the Duhamel representation of a modulated periodic wave.
\end{abstract}



\bigbreak

\section{Introduction }\label{intro}

Nonclassical viscous conservation laws
arising in multiphase fluid and solid mechanics
exhibit a rich variety of traveling wave phenomena,
including homoclinic (pulse-type) and periodic solutions
along with the standard heteroclinic (shock, or front-type)
solutions \cite{GZ,Z6,OZ1,OZ2}.
Here, we investigate stability of periodic traveling waves:
specifically, sufficient conditions for stability of the wave.
Our main result, generalizing results of Oh and Zumbrun \cite{OZ4}
in dimensions $d\ge 3$, is to show that
{\it strong spectral stability
in the sense of Schneider \cite{S1,S2,S3}
implies linearized and nonlinear
$L^1 \cap H^K\to L^\infty$ bounded stability}, for all
dimensions $d\ge 1$, and
{\it asymptotic stability} for dimensions $d\ge 2$.

More precisely, we show that small $L^1 \cap H^s$
perturbations of a planar periodic solution $u(x,t)\equiv \bar u(x_1)$
(without loss of generality taken stationary)
converge at Gaussian rate in $L^p$, $p\ge 2$ to a modulation
\be\label{mod}
\bar u(x_1-\psi(x,t)),
\ee
of the unperturbed wave,
where $x=(x_1,\tilde x)$, $\tilde x=(x_2, \dots, x_d)$, and
$\psi$ is a scalar function whose $x$- and $t$-gradients likewise
decay at least at Gaussian rate in all $L^p$, $p\ge 2$,
but which itself decays more slowly by a factor $t^{1/2}$;
in particular, $\psi$ is merely bounded in $L^\infty$ for dimension $d=1$.

The one-dimensional study of spectral stability of spatially periodic
traveling waves of systems of viscous conservation laws
was initiated by Oh and Zumbrun \cite{OZ1}
in the ``quasi-Hamiltonian'' case that the traveling-wave equation
possesses an integral of motion, and in the general case
by Serre \cite{Se1}.
An important contribution of Serre was to point out a
larger connection between the linearized dispersion relation
(the function $\lambda(\xi)$ relating spectra to wave number of the linearized
operator about the wave) near zero
and the formal Whitham averaged system obtained by
slow modulation, or WKB, approximation.

In \cite {OZ3}, this was extended to multi-dimensions,
relating the linearized dispersion relation near zero to
\ba \label{e:wkb}
\partial_t M + \sum_j \partial_{x_j}F^j &=0,   \\
\partial_t (\Omega N) + \nabla_x  (\Omega S)&=0,
\ea
where $M\in \RR^n$ denotes the average over one period,
$F^j$ the average of an associated flux,
$\Omega=|\nabla_x \Psi|\in \RR^1$ the frequency, $S=-\Psi_t/|\nabla_x \Psi|\in \RR^1$ the speed
$s$, and $N=\nabla_x \Psi/|\nabla_x \Psi|\in \RR^d$ the normal $\nu$ associated
with nearby periodic waves,
with an additional constraint
\begin{equation}\label{con}
\curl (\Omega N)=\curl \nabla_x \Psi \equiv 0.
\end{equation}
As an immediate corollary, similarly as in \cite{OZ1}, \cite{Se1} in the
one-dimensional case, this yielded as a necessary condition for
multi-dimensional stability hyperbolicity of the
averaged system \eqref{e:wkb}--\eqref{con}.

The present study is informed by but does not directly rely on
this observation relating Whitham averaging and spectral
stability properties.
Likewise, the Evans function techniques used in \cite{Se1,OZ3}
to establish this connection play no role in our analysis;
indeed, the Evans function makes no appearance here.
Rather, we rely on a direct Bloch-decomposition argument
in the spirit of Schneider \cite{S1,S2,S3}, combining sharp
linearized estimates with subtle cancellation in nonlinear
source terms arising from the modulated wave approximation.
The analytical techniques used to realize this program
are somewhat different from those of \cite{S1,S2,S3}, however,
coming instead from the theory of stability of
viscous shock fronts through a line of investigation carried
out in \cite{OZ1,OZ2,OZ3,OZ4,HoZ}.
In particular, the nonsmooth dispersion relation at $\xi=0$ typical
for convection-diffusion equations
requires
different treatment from that of \cite{S1,S2,S3}
in the reaction diffusion case; see Remark \ref{nonsmooth}.
Moreover, we detect nonlinear cancellation in the physical $x$-$t$
domain rather than the frequency domain as in \cite{S1,S2,S3}.
The main difference between the present analysis and that
of \cite{OZ4} is the
systematic incorporation of modulation approximation \eqref{mod}.

\subsection{Equations and assumptions}\label{s:equations}
Consider a parabolic system of conservation laws
\be
u_t + \sum_j f^j(u)_{x_j} = \Delta_x u,
\label{eqn:1conslaw}
\ee
$u \in {\cal U} (\hbox{open}) \in \R^n$,  $f^j \in \R^n$,
$x \in \R^d$, $d\ge 1$, $t \in \R^+$,
and a periodic traveling wave solution
\be
u=\bar{u}(x\cdot \nu -st),
\ee
of period $X$, satisfying the traveling-wave
ODE
$
\bar u''=
( \sum_j \nu_j f^j(\bar u) )'-s\bar u'
$
with boundary conditions $ \bar u(0) = \bar u(X)=:u_0.	$
Integrating, we obtain a first-order profile equation
\be
\bar u'= \sum_j \nu_j f^j(\bar u) -s \bar u -q,
\label{e:profile}
\ee
where $(u_0,q,s,\nu,X)\equiv \const$.
Without loss of generality take $\nu=e_1$, $s=0$,
so that $\bar u=\bar{u}(x_1)$ represents a stationary solution
depending only on $x_1$.

Following \cite{Se1,OZ3,OZ4}, we assume:

(H1) $f^j\in C^{K+1}$,
$K\ge [d/2]+4$.

(H2) The map $H: \,
\R \times {\cal U} \times \R \times S^{d-1} \times \R^n  \rightarrow \R^n$	
taking
$(X; a, s, \nu, q)  \mapsto u(X; a, s, \nu, q)-a$
is a submersion at point $(\bar{X}; \bar{u}(0), 0, e_1, \bar{q})$,
where $u(\cdot;\cdot)$ is the solution operator for \eqref{e:profile}.

Conditions (H1)--(H2) imply that the set of periodic solutions
in the vicinity of $\bar u$ form a
smooth $(n+d+1)$-dimensional manifold $\{\bar u^a(x\cdot \nu(a)-\alpha-s(a)t)\}$,
with $\alpha\in \RR$, $a\in \RR^{n+d}$.

\subsubsection{Linearized equations}\label{evans}

Linearizing (\ref{eqn:1conslaw}) about
$\bar{u}(\cdot)$, we obtain
\be
v_t = Lv := \Delta_x v -\sum(A^j v)_{x_j}, \label{e:lin}
\ee
where coefficients
$A^j:= Df^j(\bu)$
are now periodic functions of $x_1$.
Taking the Fourier transform in the transverse coordinate $\td{x}=
(x_2, \cdots, x_d)$, we obtain
\ba
\hv_t = L_{\td{\xi}}\hv
& = \hv_{x_1,x_1}
-(A^1 \hv)_{x_1}
 - i \sum_{j\ne 1}A^j \xi_j \hv
	- \sum_{j\ne 1} \xi_j^2  \hv,
\label{e:fourier}
\ea
where $\td{\xi}=(\xi_2, \cdots, \xi_d)$ is the transverse frequency
vector.

\subsubsection{Bloch--Fourier decomposition
and stability conditions}\label{bloch}

Following \cite{G,S1,S2,S3}, we define the family of operators
\be
L_{\xi} = e^{-i \xi_1 x_1} L_{\txi}  e^{i \xi_1 x_1}
\label{e:part}
\ee
operating on the class of $L^2$ periodic functions on $[0,X]$;
the $(L^2)$ spectrum
of $L_{\txi}$ is equal to the union of the
spectra of all $L_{\xi}$ with $\xi_1$ real with associated
eigenfunctions
\be
w(x_1, \txi,\lambda) := e^{i \xi_1 x_1} q(x_1, \x, \txi, \lambda),
\label{e:efunction}
\ee
where $q$, periodic, is an eigenfunction of $L_{\xi}$.
By continuity of spectrum,
and discreteness of the spectrum of the elliptic operators $L_\xi$ on
the compact domain $[0,X]$,
we have that the spectra of $L_{\xi}$
may be described as the union of countably many continuous
surfaces $\lambda_j(\xi)$.

Without loss of generality taking $X=1$,
recall now the {\it Bloch--Fourier representation}
\be\label{Bloch}
u(x)=
\Big(\frac{1}{2\pi }\Big)^d \int_{-\pi}^{\pi}\int_{\R^{d-1}}
e^{i\xi\cdot x}\hat u(\xi, x_1)
d\xi_1\, d\tilde \xi
\ee
of an $L^2$ function $u$, where
$\hat u(\xi, x_1):=\sum_k e^{2\pi ikx_1}\hat u(\xi_1+ 2\pi k,\tilde \xi)$
are periodic functions of period $X=1$, $\hat u(\tilde \xi)$
denoting with slight abuse of notation the Fourier transform of $u$
in the full variable $x$.
By Parseval's identity, the Bloch--Fourier transform
$u(x)\to \hat u(\xi, x_1)$ is an isometry in $L^2$:
\be\label{iso}
\|u\|_{L^2(x)}=
\|\hat u\|_{L^2(\xi; L^2(x_1))},
\ee
where $L^2(x_1)$ is taken on $[0,1]$ and $L^2(\xi)$
on $[-\pi,\pi]\times \R^{d-1}$.
Moreover, it diagonalizes the periodic-coefficient operator $L$,
yielding the {\it inverse Bloch--Fourier transform representation}
\be\label{IBFT}
e^{Lt}u_0=
\Big(\frac{1}{2\pi }\Big)^d \int_{-\pi}^{\pi}\int_{\R^{d-1}}
e^{i\xi \cdot x}e^{L_\xi t}\hat u_0(\xi, x_1)
d\xi_1\, d\tilde \xi
\ee
relating behavior of the linearized system to
that of the diagonal operators $L_\xi$.

Following \cite{OZ4}, we assume along with (H1)--(H2) the
{\it strong spectral stability} conditions:

(D1) $\sigma(L_\xi) \subset \{ \hbox{\rm Re} \lambda <0 \} $ for $\xi\ne 0$.

(D2) $\hbox{\rm Re} \sigma(L_{\xi}) \le -\theta |\xi|^2$, $\theta>0$,
for $\xi\in \R^d$ and $|\xi|$ sufficiently small.

(D3) $\lambda=0$ is a semisimple eigenvalue
of $L_{0}$ of multiplicity exactly $n+1$.\footnote{
The zero eigenspace of $L_0$ is at least $(n+1)$-dimensional
by the linearized existence theory and (H2), and hence $n+1$ is the minimal
multiplicity; see \cite{Se1,OZ3}.
As noted in \cite{OZ1,OZ3}, minimal dimension of this zero eigenspace
implies that $(M,N\Omega)$ of \eqref{e:wkb}
gives a nonsingular coordinatization
of the family of periodic traveling-wave solutions near $\bar u$.
}

For each fixed angle $\hat \xi:=\xi/|\xi|$,
expand $L_\xi= L_0 + |\xi| L^1 +|\xi|^2L^2$.
By assumption (D3) and standard spectral perturbation
theory, there exist $n+1$ smooth eigenvalues
\be\label{e:surfaces}
\lambda_j(\xi)= -i a_j(\xi)+o(|\xi|)
\ee
of $L_\xi$ bifurcating from $\lambda=0$ at $\xi=0$, where
$-i a_j$ are homogeneous degree one functions given by $|\xi|$ times the
eigenvalues of $\Pi_0 L^1|_{\kernel L_0}$, with $\Pi_0$
the zero eigenprojection of $L_0$.

Conditions (D1)--(D3) are exactly the spectral assumptions of
\cite{S1,S2,S3}, corresponding to ``dissipativity'' of the
large-time behavior of the linearized system.
As in \cite{OZ4}, we make the further nondegeneracy hypothesis:

(H3) The eigenvalues $\lambda=-i a_j(\xi)/|\xi|$ of $\Pi_0 L^1_{\kernel L_0}$
are simple.

\noindent
The functions $a_j$ may be seen to be the characteristics
associated with the Whitham averaged system \eqref{e:wkb}--\eqref{con}
linearized about the values of $M$, $S$, $N$, $\Omega$ associated
with the background wave $\bar u$; see \cite{OZ3,OZ4}.
Thus, (D1) implies weak hyperbolicity
of (\ref{e:wkb})--\eqref{con} (reality of $a_j$),
while (H1) corresponds to strict hyperbolicity.

\subsection{Main results}

With these preliminaries, we can now state our main results.

\begin{theo}\label{main}
Assuming (H1)--(H3) and (D1)--(D3),
for some $C>0$ and $\psi \in W^{K,\infty}(x,t)$,
\ba\label{eq:smallsest}
|\tilde u-\bar u(\cdot -\psi)|_{L^p}(t)&\le
C(1+t)^{-\frac{d}{2}(1-1/p)}
|\tilde u-\bar u|_{L^1\cap H^K}|_{t=0},\\
|\tilde u-\bar u(\cdot -\psi)|_{H^K}(t)&\le
C(1+t)^{-\frac{d}{4}}
|\tilde u-\bar u|_{L^1\cap H^K}|_{t=0},\\
|(\psi_t,\psi_x)|_{W^{K+1,p}}&\le
C(1+t)^{-\frac{d}{2}(1-1/p)}
|\tilde u-\bar u|_{L^1\cap H^K}|_{t=0},\\
\ea
and
\ba\label{eq:stab}
|\tilde u-\bar u|_{ L^p}(t), \; |\psi(t)|_{L^p}&\le
C(1+t)^{-\frac{d}{2}(1-\frac{1}{p}) + \frac{1}{2}}
|\tilde u-\bar u|_{L^1\cap H^K}|_{t=0}
\ea
for all $t\ge 0$, $p\ge 2$, $d= 1$,
for solutions $\tilde u$ of \eqref{eqn:1conslaw} with
$|\tilde u-\bar u|_{L^1\cap H^K}|_{t=0}$ sufficiently small.
In particular, $\bar u$ is nonlinearly bounded
$L^1\cap H^K\to L^\infty$ stable for dimension $d=1$.
\end{theo}

\begin{theo}\label{main2}
Assuming (H1)--(H3) and (D1)--(D3), for any $\eps>0$,
some $C>0$ and $\psi \in W^{K,\infty}(x,t)$,
\ba\label{eq:smallsest2}
|\tilde u-\bar u(\cdot -\psi)|_{L^p}(t)&\le
C(1+t)^{-\frac{d}{2}(1-1/p)}
|\tilde u-\bar u|_{L^1\cap H^K}|_{t=0},\\
|\tilde u-\bar u(\cdot -\psi)|_{H^K}(t)&\le
C(1+t)^{-\frac{d}{4}}
|\tilde u-\bar u|_{L^1\cap H^K}|_{t=0},\\
|(\psi_t,\psi_x)|_{W^{K+1,p}}&\le
C(1+t)^{-\frac{d}{2}(1-1/p)+\eps -\frac{1}{2}}
|\tilde u-\bar u|_{L^1\cap H^K}|_{t=0},\\
\ea
and
\ba\label{eq:stab2}
|\tilde u-\bar u|_{ L^p}(t), \; |\psi(t)|_{L^p}&\le
C(1+t)^{-\frac{d}{2}(1-\frac{1}{p}) + \eps}
|\tilde u-\bar u|_{L^1\cap H^K}|_{t=0},\\
|\tilde u-\bar u|_{ H^K}(t), \; |\psi(t)|_{H^K}&\le
C(1+t)^{-\frac{d}{4} + \eps}
|\tilde u-\bar u|_{L^1\cap H^K}|_{t=0},\\
\ea
for all $t\ge 0$, $p\ge 2$, $d= 2$,
for solutions $\tilde u$ of \eqref{eqn:1conslaw} with
$|\tilde u-\bar u|_{L^1\cap H^K}|_{t=0}$ sufficiently small.
In particular, $\bar u$ is nonlinearly
asymptotically $L^1\cap H^K\to H^K$ stable for dimension $d=2$.
\end{theo}

\br
In Theorem \ref{main2}, derivatives in $x\in\RM^2$ refer to total derivatives.  Moreover, unless specified by an
appropriate index, throughout this paper derivatives in spatial variable $x$ will always refer to the total derivative
of the function.
\er

In dimension one, Theorem \ref{main} asserts only
bounded $L^1\cap H^K \to L^\infty$ stability,
a very weak notion of stability.
The absence of decay in perturbation $\tilde u-\bar u$
indicates the delicacy of the nonlinear analysis in this case.
In particular, it is crucial to separate off the slower-decaying
modulated behavior \eqref{mod} in order to close the nonlinear iteration
argument.

\br\label{nonlinp}
\textup{
In dimension $d=1$, it is straightforward to show that the results of
Theorem \ref{main} extend to all $1\le p\le \infty$ using the pointwise
techniques of \cite{OZ2}; see Remark \ref{lowp}.
}
\er

\br\label{compare}
\textup{
The slow decay of
$|\tilde u-\bar u|_{ L^p}(t) \sim |\psi(t)|_{L^p}$
in \eqref{eq:stab}
is due to nonlinear interactions;
as shown in \cite{OZ2,OZ4}, the linearized decay rate is faster
by factor $(1+t)^{-1/2}$ (Proposition \ref{p:linstab}).
In \cite{OZ4}, it was shown that for $d\ge 3$, where linear
effects dominate behavior, \eqref{eq:stab}
may be replaced by the stronger estimate
$|\tilde u-\bar u|_{ L^p}(t), \; |\psi(t)|_{L^p}\le
C(1+t)^{-\frac{d}{2}(1-\frac{1}{p})}
|\tilde u-\bar u|_{L^1\cap H^K}|_{t=0}.$
These distinctions reflect fine details of both linearized estimates
(Section \ref{s:refined}) and nonlinear structure
(Sections \ref{s:pert}--\ref{s:cancellation})
that are not
immediately
apparent from the formal
Whitham approximation \eqref{e:wkb}--\eqref{con}.
}
\er

\subsection{Discussion and open problems}\label{s:discussion}

Linearized stability under the same assumptions, with sharp rates of decay,
was established for $d=1$ \cite{OZ2} and for $d\ge 1$ in \cite{OZ4},
along with nonlinear stability for $d\ge 3$.
Theorem \ref{main} completes this line of investigation by establishing
nonlinear stability in the critical dimensions $d=1,\,2$, a
fundamental open problem cited in \cite{OZ1,OZ4}.

This gives a generalization of the work of \cite{S1,S2,S3}
for reaction diffusion equations to the case of viscous conservation laws.
Recall that the analysis of \cite{S1,S2,S3} concerns
also multiply periodic waves, i.e., waves
that are either periodic or else constant in each coordinate direction.
It is straightforward to verify that the methods of this paper
apply essentially unchanged to this case, to give a corresponding
stability result under the analog of (H1)--(H3), (D1)--(D3),
as we intend to report further in a future work.
Likewise, the extension from the semilinear parabolic case treated
here to the general quasilinear case is straightforward,
following the treatment of \cite{OZ4}.

On the other hand, as noted in \cite{OZ2}, condition (D3) is in
the conservation law setting nongeneric, corresponding to the special
``quasi-Hamiltonian'' situation studied there;
in particular, it implies that speed is to first order constant among
the family of spatially periodic traveling-wave solutions nearby $\bar u$.
In the generic case that (D3) is violated, behavior is essentially
different \cite{OZ1,OZ2}, and
perturbations decay more slowly at the linearized level.
Nonlinear stability remains an interesting open problem in this setting.
%

Our approach to stability in the critical dimensions $d=1,\, 2$,
as suggested in \cite{OZ4}, is,
loosely following the approach of \cite{S1,S2,S3}, to subtract out
a slower-decaying part of the solution described by an appropriate
modulation equation and show that the residual decays sufficiently
rapidly to close a nonlinear iteration.
It is worth noting that the modulated approximation
$\bar u(x_1-\psi(x,t))$ of \eqref{mod} is not the full Ansatz
\be\label{modwkb}
\bar u^a(\Psi(x,t)),
\ee
$\Psi(x,t):=x_1-\psi(x,t)$,
associated with the Whitham averaged system \eqref{e:wkb}--\eqref{con},
where $\bar u^a$ is the manifold of periodic
solutions near $\bar u$ introduced below (H2), but only the translational
part not involving perturbations $a$ in the profile.
(See \cite{OZ3} for the derivation of \eqref{e:wkb}--\eqref{con}
and \eqref{modwkb}.)
That is, we don't need to separate out all variations along the
manifold of periodic solutions,
but only the special variations connected with translation invariance.

The technical reason is an asymmetry in
$y$-derivative estimates in the parts of the Green function associated
with these various modes,
something that is not apparent without a detailed study of linearized
behavior as carried out here.
This also makes sense formally, if one considers that \eqref{e:wkb}
indicates that variables $a$, $\nabla_x \Psi$ are roughly comparable,
which would suggest, by the diffusive behavior $\Psi>>\nabla_x \Psi$,
that $a$ is neglible with respect to $\Psi$.

However, note that in the case that (D3) holds, hence wave speed
is stationary along the manifold of periodic solutions,
the final equation of \eqref{e:wkb} decouples to
$(\Psi_x)_t= (\Omega N)_t=0$, and could be written as $\Psi_t=0$
in terms of $\Psi$ alone.  Hence, there is some ambiguity in
this degenerate case which of $\Psi$, $\Psi_x$
is the primary variable, and in terms of
linear behavior, the decay of variations $a$ and $\Psi$ are
in fact comparable \cite{OZ4}; in the generic case, $a$ and $\Psi_x$ are
comparable at the linearized level \cite{OZ2}.
It would be very interesting to better understand
the connection
between the Whitham averaged system (or suitable higher-order
correction) and behavior
at the nonlinear level, as explored at the linear level
in \cite{OZ3,OZ4,JZ1,JZB}.


\section{Basic linearized stability estimates}\label{linests}
We begin by recalling the basic linearized
stability estimates derived in \cite{OZ4}.
We will sharpen these afterward in Section \ref{s:refined}.
By standard spectral perturbation theory \cite{K}, the total
eigenprojection $P(\xi)$ onto the eigenspace of $L_\xi$
associated with the eigenvalues $\lambda_j(\xi)$, $j=1,\dots, n+1$
described in the introduction
is well-defined and analytic in $\xi$ for $\xi$ sufficiently small,
since these (by discreteness of the spectra of $L_\xi$) are
separated at $\xi=0$ from the rest of the spectrum of $L_0$.
Introducing a smooth cutoff function $\phi(\xi)$ that
is identically one for $|\xi|\le \eps$ and identically
zero for $|\xi|\ge 2\eps$, $\eps>0$ sufficiently small,
we split the solution operator $S(t):=e^{Lt}$ into
low- and high-frequency parts
\be\label{SI}
S^I(t)u_0:=
\Big(\frac{1}{2\pi }\Big)^d \int_{-\pi}^{\pi}\int_{\R^{d-1}}
e^{i\xi \cdot x}
\phi(\xi)P(\xi) e^{L_\xi t}\hat u_0(\xi, x_1)
d\xi_1\, d\tilde \xi
\ee
and
\be\label{SII}
S^{II}(t)u_0:=
\Big(\frac{1}{2\pi }\Big)^d \int_{-\pi}^{\pi}\int_{\R^{d-1}}
e^{i\xi \cdot x}
\big(I-\phi P(\xi)\big)
e^{L_\xi t}\hat u_0(\xi, x_1)
d\xi_1\, d\tilde \xi.
\ee

\subsection{High-frequency bounds}\label{HF}
By standard sectorial bounds \cite{He,Pa} and spectral separation
of $\lambda_j(\xi)$ from the remaining spectra of $L_\xi$,
we have trivially the exponential decay bounds
\ba\label{semigp}
\|e^{L_\xi t}(I-\phi P(\xi))f\|_{L^2([0,X])}
&\le   Ce^{-\theta t}\|f\|_{L^2([0,X])},\\
\|e^{L_\xi t}(I-\phi P(\xi))\partial_{x_1}^l f\|_{L^2([0,X])}
&\le   Ct^{-\frac{l}{2}}e^{-\theta t}\|f\|_{L^2([0,X])},\\
\|\partial_{x_1}^l e^{L_\xi t}(I-\phi P(\xi)) f\|_{L^2([0,X])}
&\le   Ct^{-\frac{l}{2}}e^{-\theta t}\|f\|_{L^2([0,X])},
\ea
for $\theta$, $C>0$, and $0\le m\le K$ ($K$ as in (H1)).
Together with (\ref{iso}), these give immediately the
following estimates.

\begin{proposition}[\cite{OZ4}]\label{p:hf}
Under assumptions (H1)--(H3) and (D1)--(D2),
for some $\theta$, $C>0$,
and all $t>0$, $2\le p\le \infty$, $0\le l\le K+1$, $0\le m\le K$,
\ba\label{SIIest}
\|\partial_x^l S^{II}(t) f\|_{L^2(x)},\;
\|S^{II}(t)\partial_x^l f\|_{L^2(x)}&\le
Ct^{-\frac{l}{2}}e^{-\theta t}\|f\|_{L^2(x)},\\
\|\partial_x^m S^{II}(t)  f\|_{L^p(x)},\;
\|S^{II}(t) \partial_x^m f\|_{L^p(x)}&\le
Ct^{-\frac{d}{2}(\frac{1}{2}-\frac{1}{p})- \frac{m}{2}}
e^{-\theta t}\|f\|_{L^2(x)},
\ea
where, again, derivatives in the variable $x\in\RM^d$ refer to total derivatives.
\end{proposition}

\begin{proof}
The first inequalities follow immediately by (\ref{iso}) and (\ref{semigp}).
The second follows for $x_1$ derivatives in the case $p=\infty$, $m=0$ by Sobolev embedding from
$$
\|S^{II}(t) f\|_{L^\infty(\tilde x; L^2(x_1))}\le
Ct^{-\frac{d-1}{4}}e^{-\theta t}\|f\|_{L^2([0,X])}
$$
and
$$
\|\partial_{x_1} S^{II}(t) f\|_{L^\infty(\tilde x; L^2(x_1))}\le
Ct^{-\frac{d-1}{4} - \frac{1}{2}
}e^{-\theta t}\|f\|_{L^2([0,X])},
$$
which follow by an application of (\ref{iso}) in the $x_1$ variable and
the Hausdorff--Young inequality $\|f\|_{L^\infty(\tilde x)}\le
\|\hat f\|_{L^1(\tilde \xi)}$ in the variable $\tilde x$.  The result for derivatives in $x_1$
and general $2\le p\le \infty$ then follows by
$L^p$ interpolation.  Finally,
the result for derivatives in $\tilde{x}$ follows from the inverse Fourier transform,
equation \eqref{SII}, and the large $|\xi|$ bound
\[
|e^{Lt}f|_{L^2(x_1)}\leq e^{-\theta|\tilde{\xi}|^2 t}|f|_{L^2(x_1)},~|\xi|\textrm{ sufficiently large},
\]
which easily follows from Parseval and the fact that $L_{\xi}$ is a relatively compact perturbation
of $\partial_x^2-|\xi|^2$.  Thus, by the above estimate we have
\begin{align*}
\|e^{Lt}\partial_{\tilde{x}}f\|_{L^2(x)}&\leq C\|e^{L_{\xi}t}|\tilde{\xi}|\hat{f}\|_{L^2(x_1,\xi)}\\
&\leq C\sup\left(e^{-\theta|\tilde{\xi}|^2 t}|\xi|\right)\|\hat{f}\|_{L^2(x_1,\xi)}\\
&\leq Ct^{-1/2}\|f\|_{L^2(x)}.
\end{align*}
A similar argument applies for $1\le m\le K$.
\end{proof}

\subsection{Low-frequency bounds}\label{LF}
Denote by
\be\label{GI}
G^I(x,t;y):=S^I(t)\delta_y(x)
\ee
the Green kernel associated with $S^I$, and
\be\label{GIxi}
[G^I_\xi(x_1,t;y_1)]:=\phi(\xi)P(\xi) e^{L_\xi t}[\delta_{y_1}(x_1)]
\ee
the corresponding kernel appearing within the Bloch--Fourier representation
of $G^I$, where the brackets on $[G_\xi]$ and $[\delta_y]$
denote the periodic extensions of these functions onto the whole line.
Then, we have the following descriptions of $G^I$, $[G^I_\xi]$,
deriving from the
spectral expansion \eqref{e:surfaces} of $L_\xi$ near $\xi=0$.

\begin{proposition}[\cite{OZ4}]\label{kernels}
Under assumptions (H1)--(H3) and (D1)--(D3),
\ba\label{Gxi}
[G^I_\xi(x_1,t;y_1)]&= \phi(\xi)\sum_{j=1}^{n+1}e^{\lambda_j(\xi)t}
q_j(\xi,x_1)\tilde q_j(\xi, y_1)^*,\\
G^I(x,t;y)&=
\Big(\frac{1}{2\pi }\Big)^d \int_{\R^{d}} e^{i\xi \cdot (x-y)}
[G^I_\xi(x_1,t;y_1)] d\xi \\
&=
\Big(\frac{1}{2\pi }\Big)^d \int_{\R^{d}}
e^{i\xi \cdot (x-y)}
\phi(\xi)
\sum_{j=1}^{n+1}e^{\lambda_j(\xi)t} q_j(\xi,x_1)\tilde q_j(\xi, y_1)^*
d\xi,
\ea
where $*$ denotes matrix adjoint, or complex conjugate transpose,
$q_j(\xi,\cdot)$ and $\tilde q_j(\xi,\cdot)$
are right and left eigenfunctions of $L_\xi$ associated with eigenvalues
$\lambda_j(\xi)$ defined in \eqref{e:surfaces},
normalized so that $\langle \tilde q_j,q_j\rangle\equiv 1$, where
$\lambda_j/|\xi|$ is a smooth function of $|\xi|$ and $\hat \xi:=\xi/|\xi|$
and $q_j$ and $\tilde q_j$ are smooth functions of
$|\xi|$, $\hat \xi:=\xi/|\xi|$, and $x_1$ or $y_1$, with
$\Re \lambda_j(\xi)\le -\theta|\xi|^2$.
\end{proposition}

\begin{proof}
Smooth dependence of $\lambda_j$ and of $q$, $\tilde q$ as functions
in $L^2[0,X]$ follow from standard spectral perturbation theory
\cite{K} using the fact that $\lambda_j$ split to first order
in $|\xi|$ as $\xi$ is varied along rays through the origin,
and that $L_{\xi}$ varies smoothly with angle $\hat \xi$.
Smoothness of $q_j$, $\tilde q_j$ in $x_1$, $y_1$ then follow from
the fact that they satisfy the eigenvalue equation for $L_\xi$,
which has smooth, periodic coefficients.
Likewise, (\ref{Gxi})(i) is immediate from the spectral decomposition
of elliptic operators on finite domains.
Substituting (\ref{GI}) into (\ref{SI})
and computing
\be\label{comp1}
\widehat{\delta_y}(\xi,x_1)=
\sum_k e^{2\pi i kx_1}\widehat{\delta_y}(\xi + 2\pi k e_1)=
\sum_k e^{2\pi i kx_1}e^{-i\xi \cdot y-2\pi i ky_1}
= e^{-i\xi \cdot y}[\delta_{y_1}(x_1)],
\ee
where the second and third equalities follow from the fact that
the Fourier transform either continuous or discrete of
the delta-function is unity, we obtain
\ba\label{GIsub}
G^I(x,t;y)&=
\Big(\frac{1}{2\pi }\Big)^d \int_{-\pi}^{\pi}\int_{\R^{d-1}}
e^{i\xi \cdot x} \phi P(\xi) e^{L_\xi t} \widehat{\delta_y}(\xi,x_1)d\xi\\
\nonumber
&=
\Big(\frac{1}{2\pi }\Big)^d \int_{-\pi}^{\pi}\int_{\R^{d-1}}
e^{i\xi \cdot (x-y)}  \phi P(\xi)e^{L_\xi t} [\delta_{y_1}(x_1)] d\xi,
\ea
yielding (\ref{Gxi})(ii) by (\ref{GIxi})(i) and the fact that $\phi$
is supported on $[-\pi,\pi]$.
\end{proof}


\begin{proposition}[\cite{OZ4}] \label{Gbds}
Under assumptions (H1)-(H3) and (D1)-(D3),
\be\label{GIest}
\sup_{y}\|G^I(\cdot, t,;y) \|_{L^p(x)},
\;
\sup_{y}\|\partial_{x,y} G^I(\cdot, t,;y) \|_{L^p(x)}
 \le  C (1+t)^{-\frac{d}{2}(1-\frac{1}{p})}
\ee
for all $2 \le p \le \infty$, $t \ge 0 $, where $C>0$ is independent of $p$.
\end{proposition}

\begin{proof}
From representation (\ref{Gxi})(ii) and $\Re \lambda_j(\xi)\le -\theta |\xi|^2$,
we obtain by the triangle inequality
\be
\|G^I\|_{L^\infty(x,y)}\le C\|e^{-\theta |\xi|^2 t} \phi(\xi)\|_{L^1(\xi)}
 \le  C (1+t)^{-\frac{d}{2}},
\ee
verifying the bounds for $p=\infty$.  Derivative bounds follow similarly,
since derivatives falling on $q_j$ or $\tilde q_j$ are harmless, whereas
derivatives falling on $e^{i\xi\cdot(x-y)}$ bring down a factor
of $\xi$, again harmless because of the cutoff function $\phi$.

To obtain bounds for $p=2$, we note that (\ref{Gxi})(ii) may be viewed
itself as a Bloch--Fourier decomposition with respect to variable
$z:=x-y$, with $y$ appearing as a parameter.
Recalling (\ref{iso}), we may thus estimate
\ba
\sup_y \|G^I(x,t;y)\|_{L^2(x)}&=
\sum_j \sup_y \|\phi(\xi) e^{\lambda_j(\xi)t}
q_j(\cdot, z_1)\tilde q_j^*(\cdot, y_1)\|_{L^2(\xi; L^2(z_1\in [0,X]))}\\
&\le
C\sum_j \sup_y \|\phi(\xi) e^{-\theta |\xi|^2t} \|_{L^2(\xi)}
\|q_j\|_{L^2(0,X)} \|\tilde q_j\|_{L^\infty(0,X)}
\\
&\le
 C (1+t)^{-\frac{d}{4}},
\ea
where we have used in a crucial way the boundedness of $\tilde q_j$;
derivative bounds follow similarly.
Finally, bounds for $2\le p\le \infty$ follow by $L^p$-interpolation.
\end{proof}

\begin{remark}\label{nonsmooth}
\textup{
In obtaining the key $L^2$-estimate, we have used in an essential
way the periodic structure of $q_j$, $\tilde q_j$.  For, viewing
$G^I$ as a general pseudodifferential expression rather than
a Bloch--Fourier decomposition, we find that the smoothness of
$q_j$, $\tilde q_j$ is not sufficient to apply standard $L^2\to L^2$
bounds of H\"ormander, which require blowup in $\xi$ derivatives
at less than the critical rate $|\xi|^{-1}$ found here; see,
e.g., \cite{H} for further discussion.
Nor do the weighted energy estimate techniques used in
\cite{S1,S2,S3} apply here, as these also rely on the property
of smoothness of $\lambda_j$, $q_j$, $\tilde q_j$ with respect
to $\xi$ at the origin $\xi=0$.
The lack of smoothness of the linearized dispersion relation
at the origin is an essential technical difference separating the
conservation law from the reaction diffusion case; see
\cite{OZ4} for further discussion.
}
\end{remark}

\begin{remark}\label{greenformula}
\textup{
Underlying the above analysis,
and also the technically rather different approach of \cite{OZ2},
is the fundamental relation
\be\label{greenform}
G(x,t;y)=
\Big(\frac{1}{2\pi }\Big)^d \int_{-\pi}^{\pi}\int_{\R^{d-1}}
e^{i\xi \cdot (x-y)}[G_\xi(x_1,t;y_1)]d\xi
\ee
which, provided $\sigma(L_\xi)$ is semisimple, yields the simple
formula
$$
G(x,t;y)=
\Big(\frac{1}{2\pi }\Big)^d \int_{-\pi}^{\pi}\int_{\R^{d-1}}
e^{i\xi \cdot (x-y)}\sum_j e^{\lambda_j(\xi)t}q_j(\xi,x_1)
\tilde q_j(\xi, y_1)^* d\xi
$$
resembling that of the constant-coefficient case,
where $\lambda_j$ runs through the spectrum of $L_\xi$.
The basic idea in both cases is to separate off the
principal part of the series involving small $\lambda_j(\xi)$
and estimate the remainder as a faster-decaying residual.
}
\end{remark}

\begin{corollary}[\cite{OZ4}]\label{Sbd}
Under assumptions (H1)--(H3) and (D1)--(D3),
for all $p\ge 2$, $t\ge 0$,
\ba\label{SIest}
\|S^I(t)f\|_{L^p},\;
\|\partial_x S^I(t)f\|_{L^p}, \;
\|S^I(t) \partial_x f\|_{L^p}
 &\le&  C (1+t)^{-\frac{d}{2}(1-\frac{1}{p})}\|f\|_{L^1}.
\ea
\end{corollary}

\begin{proof}
Immediate, from (\ref{GIest}) and the triangle inequality,
as, for example,
$$
\|S^I(t)f(\cdot )\|_{L^p}=
\Big\|\int_{\R^d}G^I(x,t;y)f(y)dy\Big\|_{L^p(x)}
\le
\int_{\R^d}\sup_y \|G^I(\cdot ,t;y)\|_{L^p}|f(y)|dy.
$$
\end{proof}

\begin{prop}[\cite{OZ4}]\label{p:linstab}
Assuming (H1)-(H3), (D1)-(D3), for some $C>0$, all $t\ge 0$, $p\ge 2$,
$0\le l\le K$,
\ba\label{Sbound}
\|S(t)\partial_{x}^l u_0\|_{L^p} &\le
C
t^{-\frac{l}{2}}
(1+t)^{-\frac{d}{2}(\frac{1}{2}-\frac{1}{p}) + \frac{l}{2}}
t^{-\frac{d}{4} - \frac{l}{2}}
\|u_0 \|_{L^1\cap L^2}.
\ea
\end{prop}

\begin{proof}
Immediate, from (\ref{SIIest}) and (\ref{SIest}).
\end{proof}

\subsection{Additional estimates}\label{s:additional}

\begin{lemma}\label{l:additional}
Assuming (H1)--(H3), (D1)--(D3), for all $t\ge 0$, $0\le l\le K$,
\be\label{nottriv}
\|\partial_x^l S^I(t)f\|_{L^p(x)},\;
\|S^I(t)\partial_x^l f\|_{L^p(x)}
\le  C(1+t)^{-\frac{d}{2}(1/2-1/p)}\|f\|_{L^2(x)}.
\ee
\end{lemma}

\begin{proof}
From boundedness of the spectral projections $P_j(\xi)=
q_j \langle \tilde q_j, \cdot\rangle$ in $L^2[0,X]$ and their derivatives,
another consequence of first-order splitting of eigenvalues
$\lambda_j(\xi)$ at the origin, we obtain
boundedness of $\phi(\xi) P(\xi)e^{L_\xi t}$
and thus, by (\ref{iso}), the global bounds
\be\label{triv}
\|\partial_x^l S^I(t)f\|_{L^2(x)},\;
\|S^I(t)\partial_x^l f\|_{L^2(x)}
\le  C\|f\|_{L^2(x)},
\ee
for all $t\ge 0$, yielding the result for $p=2$.
Moreover, by boundedness of $\tilde q$, $q$ in all $L^p(x_1)$,
we have
$$
\begin{aligned}
|\phi(\xi) P(\xi)e^{L_\xi t}\hat f(\xi, \cdot)|_{L^\infty(x_1)}
&\le Ce^{-\theta |\xi|^2t} |P(\xi) \hat f(\xi,\cdot)|_{L^\infty(x_1)}
\le Ce^{-\theta |\xi|^2t} |\hat f(\xi,\cdot)|_{L^2(x_1)},
\end{aligned}
$$
$C,\, \theta>0$, yielding by $S^If=
\Big(\frac{1}{2\pi }\Big)^d \int_{-\pi}^{\pi}\int_{\R^{d-1}}
e^{i\xi \cdot x}\phi(\xi)P(\xi) e^{L_\xi t}\hat f(\xi, x_1)
d\xi_1\, d\tilde \xi$
the bound
\ba\label{last}
\|S^{I}(t) f\|_{L^\infty(x)} &\le
\Big(\frac{1}{2\pi }\Big)^d \int_{-\pi}^{\pi}\int_{\R^{d-1}}
|\phi(\xi)P(\xi) e^{L_\xi t}\hat f(\xi, \cdot)|_{L^\infty(x_1)}
d\xi_1\, d\tilde \xi \\
&\le
\Big(\frac{1}{2\pi }\Big)^d \int_{-\pi}^{\pi}\int_{\R^{d-1}}
C\phi(\xi) e^{-\theta |\xi|^2t}|\hat f(\xi,\cdot)|_{L^2(x_1)}
d\xi_1\, d\tilde \xi \\
&\le C|\phi(\xi) e^{-\theta |\xi|^2t}|_{L^2(\xi)} |\hat f|_{L^2(\xi, x_1)}\\
&= C(1+t)^{-\frac{d}{4} } \|f\|_{L^2([0,X])},
\ea
yielding the result for $p=\infty$, $l=0$.
The result for $p=\infty$, $1\le l\le K$ follows by
a similar argument.
The result for general $2\le p\le \infty$ then follows by
$L^p$ interpolation between $p=2$ and $p=\infty$.
\end{proof}

By Riesz--Thorin interpolation between (\ref{nottriv}) and (\ref{SIest}),
we obtain the following, apparently sharp bounds between various
$L^q$ and $L^p$.\footnote{
The inclusion of general $p\ge 2$ in Lemma \ref{l:additional}
repairs an omission in \cite{OZ4}, where the
bounds \eqref{RT} were stated but not used.}

\begin{corollary}\label{RT}
Assuming (H0)--(H3) and (D1)--(D3),
for all $1\le q\le 2\le p$, $t\ge 0$,
$0\le l\le K$,
\ba\label{RTest}
\|\partial_x^l S^I(t)f\|_{L^p}, \;
\|S^I(t) \partial_x^l f\|_{L^p}
 &\le  C (1+t)^{-\frac{d}{2}(\frac{1}{q}-\frac{1}{p})}\|f\|_{L^q}.
\ea
\end{corollary}

\begin{prop}\label{p:addlinstab}
Assuming (H1)-(H3), (D1)-(D3), for some $C>0$, all $t\ge 0$,
$1\le q\le 2\le p$, and $0\le l\le K$,
\ba\label{addSbound}
\|S(t)\partial_{x}^l u_0\|_{L^p} &\le
C(1+t)^{-\frac{d}{2}(\frac{1}{2}-\frac{1}{p})+\frac{l}{2}}
t^{-\frac{d}{2}(\frac{1}{q}-\frac{1}{2}) - \frac{l}{2} }
\|u_0 \|_{L^q\cap L^2}.
\ea
\end{prop}

\begin{proof}
Immediate, from (\ref{SIIest}) and (\ref{RT}).
\end{proof}

%
%

\section{Refined linearized estimates}\label{s:refined}
The bounds of Proposition \ref{p:linstab} are sufficient to
establish nonlinear stability and asymptotic behavior in
dimensions $d\ge 3$, as shown in \cite{OZ4}.
However, they are not sufficient in the critical dimensions
$d=1,2$; see Remark 1, Section 7 of \cite{OZ4}.
Comparison with standard diffusive stability arguments as in \cite{Z7}
show that this is due to the fact that the full solution operator
$|S(t)\partial_x |$ decays no faster
than $S(t)$, or, equivalently, $G_y$ no faster than $G$.

Following the basic strategy introduced in \cite{ZH,Z1,MaZ2,MaZ4}
in the context of viscous shock waves,
we now perform a refined linearized estimate separating
slower-decaying translational modes from a
faster-decaying ``good'' part of the solution operator.
This will be used in Section \ref{s:nonlin} in combination with
certain nonlinear cancellation estimates to show convergence
to the modulated approximation \eqref{mod} at a faster rate
sufficient to close the nonlinear iteration.

The key to this decomposition is the following observation.

\begin{lemma}\label{blochfacts}
Assuming (H1)--(H3), (D1)--(D3), let $\lambda_j(\xi/|\xi|, \xi)$,
$q_j(\xi/|\xi|, \xi, \cdot)$, $\tilde q_j(\xi/|\xi|, \xi, \cdot)$
denote the eigenvalues and associated right and left eigenfunctions of
$L_\xi$, with $q_j$, $\tilde q_j$ smooth functions of $\xi/|\xi|$ and
$|\xi|$ as noted in Prop. \ref{kernels}.
Then, without loss of generality, $q_1(\omega, 0, \cdot)\equiv \bar u'$,
while $\tilde q_j(\omega, 0, \cdot)$ for $j\ne 1$ are constant functions
depending only on angle $\omega=\xi/|\xi|$.
\end{lemma}

\begin{proof}
Expanding $L_\xi=L_0 + |\xi|L^1_{\xi/|\xi|}+ |\xi|^2L^2_{\xi/|\xi|}$
as in the introduction, consider the continuous family of spectral
perturbation problems in $|\xi|$ indexed by angle $\omega=\xi/|\xi|$.
Then, both facts follow by standard perturbation theory \cite{K} using the
observations that $\bar u'$ is in the right kernel of $L_0$ and
constant functions $c$ are in the left kernel of $L_0$, with
$$
\langle c,L^1 \bar u'\rangle=
\langle c,( \omega_1(2\partial_{x_1}-A_1) -\sum_{j\ne 1} \omega_j A_j)
)\bar u'\rangle=
\langle c, \omega_1\partial_{x_1}^2 \bar u  -\sum_{j\ne 1} \omega_j
\partial_{x_1} f^j(\bar u)\rangle
\equiv 0,
$$
where $\langle \cdot, \cdot\rangle$ denotes $L^2(x_1)$ inner product
on the interval $x_1\in [0,X]$,
that the dimension of $\ker L_0$ by assumption is $(n+1)$, so
that the orthogonal complement of $\bar u'$ in $\kernel L_0$ is dimension $n$
so exactly the set of constant functions,
and that by (H3) the functions $q_j(\omega, 0,\cdot)$
and $\tilde q_j(\omega,0)$ are right and left eigenfunctions
of $\Pi_0 L^1|_{\ker L_0}$ ($\Pi_0$ as earlier denoting the zero
eigenprojection associated with $L_0$).
\end{proof}

\br\label{wkbrmk}
\textup{
The key observation of Lemma \ref{blochfacts} can be motivated
by the form of the Whitham averaged system \eqref{e:wkb}.
For, recalling (Section \ref{s:discussion})
that (D3) implies that speed $s$ is stationary to first order
at $\bar u$ along the manifold of nearby periodic solutions,
we find that the last equation of \eqref{e:wkb} reduces to
$(\nabla_x \Psi)_t=0$, i.e., the equation for the translational
variation $\Psi$ decouples from the equations for variations
in other modes.
This corresponds heuristically to the fact derived above that
the translational mode $\bar u'(x_1)$ decouples in the first-order
eigenfunction expansion.
}
\er

\begin{cor}\label{greenbds}
Under assumptions (H1)--(H3), (D1)--(D3),
the Green function $G(x,t;y)$ of \eqref{e:lin} decomposes as
$G=E+\tilde G$,
\be\label{E}
E=\bar u'(x)e(x,t;y),
\ee
where, for some $C>0$, all $t>0$, $1\le q\le 2\le p\le \infty$, $0\le j,k, l$,
$j+l\le K$, $1\le r\le 2$,
\ba\label{sheatbds}
\Big|\int_{-\infty}^{+\infty} \tilde G(x,t;y)f(y)dy\Big|_{L^p(x)}&\le
C (1+t)^{-\frac{d}{2}(1/2-1/p)} t^{-\frac{1}{2}(1/q-1/2)}
|f|_{L^q\cap L^2},\\
\Big|\int_{-\infty}^{+\infty} \partial_y^r \tilde G(x,t;y)f(y)dy\Big|_{L^p(x)}&\le
C (1+t)^{-\frac{d}{2}(1/2-1/p)-\frac{1}{2}+\frac{r}{2}} \\
&\quad \times
t^{-\frac{d}{2}(1/q-1/2)-\frac{r}{2}} |f|_{L^q\cap L^2},\\
\Big|\int_{-\infty}^{+\infty} \partial_t^r \tilde G(x,t;y)f(y)dy\Big|_{L^p(x)}&\le
C (1+t)^{-\frac{d}{2}(1/2-1/p)-\frac{1}{2}+r}
\\ &\quad \times t^{-\frac{d}{2}(1/q-1/2)-r} |f|_{L^q\cap L^2}.\\
\ea
\ba\label{etbds}
\Big|\int_{-\infty}^{+\infty} \partial_x^j\partial_t^k\partial_y^l e(x,t;y)f(y)dy\Big|_{L^p}
&\le
(1+t)^{-\frac{d}{2}(1/q-1/p) -\frac{(j+k)}{2} }|f|_{L^q}.
\ea
Moreover, $e(x,t;y)\equiv 0$ for $t\le 1$.
\end{cor}

\begin{proof}
We first treat the simpler case $q=1$.
Recalling that
\ba\label{Gnew}
G^I(x,t;y)&=
\Big(\frac{1}{2\pi }\Big)^d \int_{\R^{d}}
e^{i\xi \cdot (x-y)}
\phi(\xi)
\sum_{j=1}^{n+1}e^{\lambda_j(\xi)t} q_j(\xi,x_1)\tilde q_j(\xi, y_1)^*
d\xi,
\ea
define
\ba\label{enew}
\tilde e(x,t;y)&=
\Big(\frac{1}{2\pi }\Big)^d \int_{\R^{d}}
e^{i\xi \cdot (x-y)}
\phi(\xi)
e^{\lambda_1(\xi)t} \tilde q_1(\xi, y_1)^* d\xi,
\ea
so that
\ba\label{Gdiff}
G^I(x,t;y)-\bar u'(x_1)\tilde e(x,t;y)&=
\Big(\frac{1}{2\pi }\Big)^d \int_{\R^{d}}
e^{i\xi \cdot (x-y)}
\phi(\xi)
\sum_{j=2}^{n+1}e^{\lambda_j(\xi)t} q_j(\xi/|\xi|, 0,x_1)\tilde q_j(\xi, y_1)^*
d\xi\\
&+
\Big(\frac{1}{2\pi }\Big)^d \int_{\R^{d}}
\sum_{j=1}^{n+1}e^{i\xi \cdot (x-y)}
\phi(\xi)
e^{\lambda_j(\xi)t} O(|\xi|) d\xi.
\ea

Noting, by Lemma \ref{blochfacts}, that
$\partial_y \tilde q(\omega, 0, y)\equiv \const$
for $j\ne 1$, we have therefore
\ba\label{Gdiffy}
\partial_y^r(G^I(x,t;y)-\bar u'(x_1)\tilde e(x,t;y))&=
\Big(\frac{1}{2\pi }\Big)^d \int_{\R^{d}}
e^{i\xi \cdot (x-y)}
\phi(\xi)
\sum_{j=1}^{n+1}e^{\lambda_j(\xi)t}O(|\xi|) d\xi,\\
\ea
which readily gives
\ba\label{Gdiffyest}
|\partial_y^r(G^I(x,t;y)-\bar u'(x_1)\tilde e(x,t;y))|_{L^p}&\le
C(1+t)^{-\frac{d}{2}(1-1/p) -\frac{1}{2}},
\ea
$p\ge 2$, by the same argument used to prove \eqref{GIest}, and
similarly
\ba\label{Gdiffyest2}
|\partial_t^r(G^I(x,t;y)-\bar u'(x_1)\tilde e(x,t;y))|_{L^p}&\le
c(1+t)^{-\frac{d}{2}(1-1/p) -\frac{1}{2}}.
	\ea
These yield \eqref{sheatbds} by the triangle inequality.

Defining $e(x,t;y):= \chi(t)\tilde e(x,t;y)$, where
$\chi$ is a smooth cutoff function
such that $\chi(t)\equiv 1$ for $t\ge 2$ and $\chi(t)\equiv 0$ for $t\le 1$,
and setting $\tilde G:=G-\bar u'(x_1)e(x,t;y)$,
we readily obtain the estimates \eqref{sheatbds} by combining
\eqref{Gdiffyest2} with bound \eqref{SIIest} on $G^{II}$.
Bounds \eqref{etbds} follow from \eqref{enew} by the argument
used to prove \eqref{GIest}, together with the observation that
$x$- or $t$-derivatives bring down factors of $|\xi|$,
followed again by an application of the triangle inequality.

The cases $1\le q\le 2$ follow similarly, by the arguments used
to prove \eqref{nottriv} and \eqref{RT}.
\end{proof}

\br\label{lowp}
\textup{
Despite their apparent complexity, the above bounds
may be recognized as essentially just the
standard diffusive bounds satisfied for the heat equation \cite{Z7}.
For dimension $d=1$, it may be shown using pointwise techniques as
in \cite{OZ2} that the bounds of Corollary
\ref{greenbds} extend to all $1\le q\le p\le \infty$.
}
\er

Note the strong analogy between the Green function decomposition
of Corollary \ref{greenbds} and that of \cite{MaZ3,Z4}
in the viscous shock case.
We pursue this analogy further in the nonlinear analysis of
the following sections,
combining the ``instantaneous tracking'' strategy of
\cite{ZH,Z1,Z4,Z7,MaZ2,MaZ4} with a type of
cancellation estimate introduced in \cite{HoZ}.

\section{Nonlinear stability in dimension one}\label{s:nonlin}

For clarity, we carry out the nonlinear stability analysis
in detail in the most difficult, one-dimensional, case,
indicating afterward by a few brief remarks the
extension to $d= 2$.
Hereafter, take $x\in \RR^1$, dropping the indices
on $f^j$ and $x_j$ and writing $u_t+f(u)_x=u_{xx}$.

\subsection{Nonlinear perturbation equations}\label{s:pert}

Given a solution $\tilde u(x,t)$ of \eqref{eqn:1conslaw},
define the nonlinear perturbation variable
\be\label{pertvar}
v=u-\bar u=
\tilde {u}(x+\psi(x,t))-\bar u(x),
\ee
where
\be\label{uvar}
u(x,t):=\tilde {u}(x+\psi(x,t))
\ee
and $\psi:\RM\times\RM\to\RM$ is to be chosen later.

\begin{lem}\label{4.1}
For $v$, $u$ as in \eqref{pertvar},\eqref{uvar},
\begin{equation}\label{eqn:1nlper}
u_t+f(u)_{x}-u_{xx}=\left(\partial_t-L\right)\bar{u}'(x_1)\psi(x,t)
+\partial_x R + (\partial_t+\partial_x^2)  S ,
\ee
where
\[
R:= v\psi_t + v\psi_{xx}+  (\bar u_x +v_x)\frac{\psi_x^2}{1+\psi_x}
= O(|v|(|\psi_t|+|\psi_{xx}|) +\Big(\frac{|\bar u_x|+|v_x|}{1-|\psi_x|} \Big)|\psi_x|^2)
\]
and
\[
S:=- v\psi_x =O(|v|(|\psi_x|).
\]
\end{lem}

\begin{proof}
To begin, notice from the definition of $u$ in \eqref{uvar} we have by a
straightforward computation
\begin{align*}
u_t(x,t)&=\tilde{u}_x(x+\psi(x,t),t)\psi_t(x,t)+\tilde{u}_t(x+\psi,t)\\
f(u(x,t))_x&=df(\tilde{u}(x+\psi(x,t),t))\tilde{u}_x(x+\psi,t)\cdot(1+\psi_x(x,t))
\end{align*}
and
\begin{align*}
u_{xx}(x,t)&=\left(\tilde{u}_x(x+\psi(x,t),t)\cdot(1+\psi_x(x,t))\right)_x\\
&=\tilde{u}_{xx}(x+\psi(x,t),t)\cdot(1+\psi_x(x,t))+\left(\tilde{u}_x(x+\psi(x,t),t)\cdot\psi_x(x,t)\right)_x.
\end{align*}
Using the fact that
$\tilde u_t + df(\tilde{u})\tilde{u}_x-\tilde{u}_{xx}=0$, it follows that
\ba\label{altform}
u_t+f(u)_{x}-u_{xx}&=\tilde{u}_x\psi_t+df(\tilde{u})\tilde{u}_x\psi_x-\tilde{u}_{xx}\psi_x-\left(\tilde{u}_x\psi_x\right)_x\\
&= \tilde u_x \psi_t
-\tilde u_{t} \psi_x - (\tilde u_x \psi_x)_x
\ea
where it is understood that derivatives of $\tilde u$ appearing
on the righthand side
are evaluated at $(x+\psi(x,t),t)$.
Moreover, by another direct calculation,
using the fact that $L(\bar{u}'(x))=0$ by translation invariance,
we have
\begin{align*}
\left(\partial_t-L\right)\bar{u}'(x)\psi&=\bar{u}_x\psi_t
-\bar{u}_t\psi_{x} -(\bar{u}_x\psi_{x})_{x}.
\end{align*}
Subtracting, and using the facts that,
by differentiation of $(\bar u+ v)(x,t)= \tilde u(x+\psi,t)$,
\ba\label{keyderivs}
\bar u_x + v_x&= \tilde u_x(1+\psi_x),\\
\bar u_t + v_t&= \tilde u_t + \tilde u_x\psi_t,\\
\ea
so that
\ba\label{solvedderivs}
\tilde u_x-\bar u_x -v_x&=
-(\bar u_x+v_x) \frac{\psi_x}{1+\psi_x},\\
\tilde u_t-\bar u_t -v_t&=
-(\bar u_x+v_x) \frac{\psi_t}{1+\psi_x},\\
\ea
we obtain
\begin{align*}
u_t+ f(u)_{x} - u_{xx}&=
(\partial_t-L)\bar{u}'(x)\psi
+v_x\psi_t
- v_t \psi_x - (v_x\psi_x)_x
+ \Big((\bar u_x +v_x)\frac{\psi_x^2}{1+\psi_x} \Big)_x,
\end{align*}
yielding \eqref{eqn:1nlper} by
$v_x\psi_t - v_t \psi_x = (v\psi_t)_x-(v\psi_x)_t$
and
$(v_x\psi_x)_x= (v\psi_x)_{xx} - (v\psi_{xx})_{x} $.
\end{proof}

\begin{cor}
The nonlinear residual $v$ defined in \eqref{pertvar} satisfies
\be\label{veq}
v_t-Lv=\left(\partial_t-L\right)\bar{u}'(x_1)\psi
-Q_{x}+ R_x +(\partial_t+\partial_x^2)S,
\ee
where
\be\label{eqn:Q}
Q:=f(\tilde{u}(x+\psi(x,t),t))-f(\bar{u}(x))-df(\bar{u}(x))v=\mathcal{O}(|v|^2),
\ee
\be\label{eqn:R}
R:= v\psi_t + v\psi_{xx}+  (\bar u_x +v_x)\frac{\psi_x^2}{1+\psi_x},
\ee
and
\be\label{eqn:S}
S:= -v\psi_x =O(|v|(|\psi_x|).
\ee
\end{cor}

\begin{proof}
Taylor expansion comparing \eqref{eqn:1nlper} and
$\bar u_t + f(\bar u)_x-\bar u_{xx}=0$.
\end{proof}

\subsection{Cancellation estimate}\label{s:cancellation}

Our strategy in writing \eqref{veq} is motivated by the following
basic cancellation principle.

\begin{prop}[\cite{HoZ}]\label{p:cancellation}
For any $f(y,s)\in L^p \cap C^2$ with $f(y,0)\equiv 0$, there holds
\be\label{e:cancel}
\int^t_0 \int G(x,t-s;y) (\partial_s - L_y)f(y,s) dy\,ds
= f(x,t).
\ee
\end{prop}

\begin{proof} Integrating the left hand side by parts, we obtain
\be
\int G(x,0;y)f(y,t)dy - \int G(x,t;y)f(y,0)dy
+ \int^t_0 \int
(\partial_t - L_y)^*G(x,t-s;y) f(y,s)dy\, ds.
\label{5.53.2}
\ee
Noting that, by duality,
$$
(\partial_t - L_y)^* G(x,t-s;y) = \delta(x-y) \delta(t-s),
$$
$\delta(\cdot)$ here denoting the Dirac delta-distribution,
we find that the third term on the righthand side
vanishes in \eqref{5.53.2}, while,
because $G(x,0;y) = \delta(x-y)$, the first term is simply $f(x,t)$.
The second term vanishes by $f(y,0)\equiv 0$.
\end{proof}

\br\label{trackrmk}
\textup{
For $\psi=\psi(t)$, term $(\partial_t-L)\bar u' \psi$
in \eqref{veq}
reduces to the term $\dot\psi(t)\bar u'(x)$ appearing
in the shock wave case \cite{ZH,Z1,Z4,Z7,MaZ2,MaZ4}.
}
\er

\subsection{Nonlinear damping estimate}


\begin{proposition}\label{damping}
Let $v_0\in H^K$ ($K$ as in (H1)), and suppose that
for $0\le t\le T$, the $H^K$ norm of $v$
and the $H^K(x,t)$ norms of $\psi_t$ and $\psi_x$
remain bounded by a sufficiently small constant.
There are then constants $\theta_{1,2}>0$ so that, for all $0\leq t\leq T$,
\begin{equation}\label{Ebounds}
|v(t)|_{H^K}^2 \leq C e^{-\theta_1 t} |v(0)|^2_{H^K} +
C \int_0^t e^{-\theta_2(t-s)} \left(|v|_{L^2}^2 +
|(\psi_t, \psi_x)|_{H^K(x,t)}^2 \right) (s)\,ds.
\end{equation}
\end{proposition}

\begin{proof}
Subtracting from the equation \eqref{altform} for $u$
the equation for $\bar u$, we may write the
nonlinear perturbation equation as
\ba\label{vperturteq}
v_t + (df(\bar u)v)_x-v_{xx}= Q(v)_x
+ \tilde u_x \psi_t -\tilde u_{t} \psi_x - (\tilde u_x \psi_x)_x,
\ea
where it is understood that derivatives of $\tilde u$ appearing
on the righthand side
are evaluated at $(x+\psi(x,t),t)$.
Using \eqref{solvedderivs} to replace $\tilde u_x$ and
$\tilde u_t$ respectively by
$\bar u_x + v_x -(\bar u_x+v_x) \frac{\psi_x}{1+\psi_x}$
and
$\bar u_t + v_t -(\bar u_x+v_x) \frac{\psi_t}{1+\psi_x}$,
and moving the resulting $v_t\psi_x$ term to the lefthand side
of \eqref{vperturteq}, we obtain
\ba\label{vperturteq2}
(1+\psi_x) v_t -v_{xx}&=
-(df(\bar u)v)_x+ Q(v)_x
+ \bar u_x \psi_t
\\ &\quad
- ((\bar u_x+v_x)  \psi_x)_x
+ \Big((\bar u_x+v_x) \frac{\psi_x^2}{1+\psi_x}\Big)_x.
\ea

Taking the $L^2$ inner product in $x$ of
$\sum_{j=0}^K \frac{\partial_x^{2j}v}{1+\psi_x}$
against (\ref{vperturteq2}), integrating by parts,
and rearranging the resulting terms,
we arrive at the inequality
\[
\partial_t |v|_{H^K}^2(t) \leq -\theta |\partial_x^{K+1} v|_{L^2}^2 +
C\left( |v|_{H^K}^2
+
|(\psi_t, \psi_x)|_{H^K(x,t)}^2 \right)
,
\]
for some $\theta>0$, $C>0$, so long as $|\tilde u|_{H^K}$ remains bounded,
and $|v|_{H^K}$ and $|(\psi_t,\psi_x)|_{H^K(x,t)}$ remain sufficiently small.
Using the Sobolev interpolation
$
|v|_{H^K}^2 \leq  |\partial_x^{K+1} v|_{L^2}^2 + \tilde{C} | v|_{L^2}^2
$
for $\tilde{C}>0$ sufficiently large, we obtain
$
\partial_t |v|_{H^K}^2(t) \leq -\tilde{\theta} |v|_{H^K}^2 +
C\left( |v|_{L^2}^2+ |(\psi_t, \psi_x)|_{H^K(x,t)}^2 \right)
$
from which (\ref{Ebounds}) follows by Gronwall's inequality.
\end{proof}

\subsection{Integral representation/$\psi$-evolution scheme}

By Proposition \ref{p:cancellation},
we have, applying Duhamel's principle to \eqref{veq},
\ba\label{prelim}
  v(x,t)&=\int^\infty_{-\infty}G(x,t;y)v_0(y)\,dy  \\
  &\quad
  + \int^t_0 \int^\infty_{-\infty} G(x,t-s;y)
  (-Q_y+ R_x + S_t + S_{yy} ) (y,s)\,dy\,ds
+ \psi (t) \bar u'(x).
\ea
Defining $\psi$ implicitly as
\ba
  \psi (x,t)& =-\int^\infty_{-\infty}e(x,t;y) u_0(y)\,dy \\
&\quad
  -\int^t_0\int^{+\infty}_{-\infty} e(x,t-s;y)
  (-Q_y+ R_x + S_t + S_{yy} ) (y,s)\,dy\,ds ,
 \label{psi}
\ea
following \cite{ZH,Z4,MaZ2,MaZ3},
where $e$ is defined as in \eqref{E},
and substituting in \eqref{prelim} the decomposition $G=\bar u'(x)e +  \tilde G$ of Corollary \ref{greenbds},
we obtain the {\it integral representation}
\ba \label{u}
  v(x,t)&=\int^\infty_{-\infty} \tilde G(x,t;y)v_0(y)\,dy \\
&\quad
  +\int^t_0\int^\infty_{-\infty}\tilde G(x,t-s;y)
  (-Q_y+ R_x + S_t + S_{yy} ) (y,s)\,dy\,ds ,
\ea
and, differentiating (\ref{psi}) with respect to $t$,
and recalling that
$e(x,s;y)\equiv 0$ for $s \le 1 $,
\ba \label{psidot}
   \partial_t^j\partial_x^k \psi (x,t)&=-\int^\infty_{-\infty}\partial_t^j\partial_x^k
e(x,t;y) u_0(y)\,dy \\
&\quad
  -\int^t_0\int^{+\infty}_{-\infty} \partial_t^j\partial_x^k
e(x,t-s;y)
  (-Q_y+ R_x + S_t + S_{yy} ) (y,s)\,dy\,ds .
  \ea

Equations \eqref{u}, \eqref{psidot}
together form a complete system in the variables $(v,\partial_t^j \psi,
\partial_x^k\psi)$,
$0\le j\le 1$, $0\le k\le K$,
from the solution of which we may afterward recover the
shift $\psi$ via \eqref{psi}.
From the original differential equation \eqref{veq}
together with \eqref{psidot},
we readily obtain short-time existence and continuity with
respect to $t$ of solutions
$(v,\psi_t, \psi_x)\in H^K$
by a standard contraction-mapping argument based on \eqref{Ebounds},
\eqref{psi}, and and \eqref{etbds}.

%

\subsection{Nonlinear iteration}

Associated with the solution $(u, \psi_t, \psi_x)$ of integral system
\eqref{u}--\eqref{psidot}, define
\ba\label{szeta}
\zeta(t)&:=\sup_{0\le s\le t}
 |(v, \psi_t,\psi_x)|_{H^K}(s)(1+s)^{1/4} .
\ea

\bl\label{sclaim}
For all $t\ge 0$ for which $\zeta(t)$ is finite, some $C>0$,
and $E_0:=|u_0|_{L^1\cap H^K}$,
\be\label{eq:sclaim}
\zeta(t)\le C(E_0+\zeta(t)^2).
\ee
\el

\begin{proof}
By \eqref{eqn:R}--\eqref{eqn:S} and definition \eqref{szeta},
\ba\label{sNbds}
|(Q,R,S)|_{L^1\cap L^\infty}
&\le |(v,v_x,\psi_t,\psi_x)|_{L^2}^2+
|(v,v_x,\psi_t,\psi_x)|_{L^\infty}^2
\le C\zeta(t)^2 (1+t)^{-\frac{1}{2}},\\
\ea
so long as $|\psi_x|\le |\psi_x|_{H^K}\le \zeta(t)$ remains small,
and likewise (using the equation to bound $t$ derivatives in terms
of $x$-derivatives of up to two orders)
\ba\label{sNbds2}
|(\partial_t+\partial_x^2)S|_{L^1\cap L^\infty}
&\le |(v,\psi_x)|_{H^2}^2
+ |(v,\psi_x)|_{W^{2,\infty}}^2
\le C\zeta(t)^2 (1+t)^{-\frac{1}{2}}.\\
\ea

Applying Corollary \ref{greenbds} with $q=1$, $d=1$ to representations
\eqref{u}--\eqref{psidot}, we obtain for any $2\le p<\infty$
\ba\label{sest}
|v(\cdot,t)|_{L^p(x)}& \le
C(1+t)^{-\frac{1}{2}(1-1/p)}E_0 \\
&\quad +
C\zeta(t)^2\int_0^{t} (1+t-s)^{-\frac{1}{2}(1/2-1/p)}(t-s)^{-\frac{3}{4}}
(1+s)^{-\frac{1}{2}}ds\\
&
\le
 C(E_0+\zeta(t)^2) (1+t)^{-\frac{1}{2}(1-1/p)}
\ea
and
\ba\label{sestad}
|(\psi_t,\psi_x)(\cdot, t)|_{W^{K,p}}& \le
C(1+t)^{-\frac{1}{2}}E_0 +
C\zeta(t)^2\int_0^{t} (1+t-s)^{-\frac{1}{2}(1-1/p)-1/2}
(1+s)^{-\frac{1}{2}}ds \\
&
\le
 C(E_0+\zeta(t)^2) (1+t)^{-\frac{1}{2}(1-1/p)}.
\ea
Using \eqref{Ebounds} and \eqref{sest}--\eqref{sestad},
we obtain
$|v(\cdot,t)|_{H^K(x)} \le
 C(E_0+\zeta(t)^2) (1+t)^{-\frac{1}{4}}$.
Combining this with \eqref{sestad}, $p=2$, rearranging, and recalling
definition \eqref{szeta}, we obtain \eqref{sclaim}.
\end{proof}

\begin{proof}[Proof of Theorem \ref{main}]
By short-time $H^K$ existence theory,
$\|(v,\psi_t,\psi_x)\|_{H^{K}}$ is continuous so long as it
remains small, hence $\eta$ remains
continuous so long as it remains small.
By \eqref{sclaim}, therefore,
it follows by continuous induction that
$\eta(t) \le 2C \eta_0$ for $t \ge0$, if $\eta_0 < 1/ 4C$,
yielding by (\ref{szeta}) the result (\ref{eq:smallsest}) for $p=2$.
Applying \eqref{sest}--\eqref{sestad}, we obtain
(\ref{eq:smallsest}) for $2\le p\le p_*$ for any $p_*<\infty$,
with uniform constant $C$.
Taking $p_*>4$ and estimating
$$
|Q|_{L^2}, \, |R|_{L^2}, \, |S|_{L^2}(t)
\le |(v,\psi_t,\psi_x)|_{L^4}^2\le CE_0(1+t)^{-\frac{3}{4}}
$$
in place of the weaker \eqref{sNbds},
then applying Corollary \ref{greenbds} with $q=2$, $d=1$,
we obtain finally \eqref{eq:smallsest} for $2\le p\le \infty$,
by a computation similar \eqref{sest}--\eqref{sestad};
we omit the details of this final bootstrap argument.
Estimate \eqref{eq:stab} then follows using \eqref{etbds} with
$q=d=1$, by
\ba\label{sesta}
|\psi(t)|_{L^p}& \le
C E_0 +
C\zeta(t)^2\int_0^{t} (1+t-s)^{-\frac{1}{2}(1-1/p)}
(1+s)^{-\frac{1}{2}}ds
\le
C(1+t)^{\frac{1}{2p}}(E_0+\zeta(t)^2),
\ea
together with the fact that
$ \tilde u(x,t)-\bar u(x)= v(x-\psi,t)+ (\bar u(x)-\bar u(x-\psi), $
so that $|\tilde u(\cdot, t)-\bar u|$ is controlled
by the sum of $|v|$ and
$|\bar u(x)-\bar u(x-\psi)|\sim |\psi|$.
This yields stability for $|u-\bar u|_{L^1\cap H^K}|_{t=0}$
sufficiently small, as described in the final line of the theorem.
\end{proof}

\section{Nonlinear stability in dimension two}\label{s:nonlin2}
We now briefly sketch the extension to dimension $d=2$.
%
%
Given a solution $\tilde u(x,t)$ of \eqref{eqn:1conslaw},
define the nonlinear perturbation variable
\be\label{pertvar2}
v=u-\bar u= \tilde {u}(x_1+\psi(x,t),x_2,t)-\bar u(x_1),
\ee
where
\be\label{uvar2}
u(x,t):=\tilde {u}(x_1+\psi(x,t),t)
\ee
and $\psi:\RM^d\times\RM\to\RM$ is to be chosen later.

\begin{lem}\label{l:pert}
For $v$, $u$ as in \eqref{uvar2},
\begin{equation} \label{eqn:nlper}
u_t+\sum_{j=1}^df^j(u)_{x_j}-\sum_{j=1}^du_{x_jx_j}=
\left(\partial_t-L\right)\bar{u}'(x_1)\psi(x,t)
+\sum_{j=1}^d\partial_{x_j} R_j + \partial_t  S  + T,
\end{equation}
where
\[
R_j = O((|v,\psi_t,\psi_{x})| |(v,v_x,\psi_t,\psi_{x})|),\quad
S:=-v\psi_{x_1}=(|v|(|\psi_x|),\quad
T:= O(|\psi_x|^3 + |(v,\psi_x)||\psi_{xx}|).
\]
\end{lem}

\begin{proof}
Similarly as in the proof of Lemma \ref{4.1},
it follows by a straightforward computation Using the fact that
$\tilde u_t + \sum_j df^j(\tilde{u})\tilde{u}_{x_j}
-\sum_j \tilde{u}_{x_jx_j}=0$, it follows that
\ba\label{altform2}
u_t + \sum_j df^j(u){u}_{x_j} -\sum_j {u}_{x_jx_j}&=
\tilde{u}_{x_1}\psi_t
-\tilde{u}_{t}\psi_{x_1}
+ \sum_{j\ne 1}df^j(\tilde u)\tilde u_{x_1}\psi_{x_j}\\
&\quad - \sum_{j\ne 1}\tilde u_{x_jx_1}\psi_{x_j}
- \sum_{j}(\tilde u_{x_1}\psi_{x_j} )_{x_j},
\ea
where it is understood that derivatives of $\tilde u$ appearing
on the righthand side
are evaluated at $(x+\psi(x,t),t)$.
Moreover, by another direct calculation,
using the fact that $L(\bar{u}'(x_1))=0$ by translation invariance,
we have
\begin{align*}
\left(\partial_t-L\right)\bar{u}'(x_1)\psi&=\bar{u}_{x_1}\psi_t
-\bar u_{t}\psi_{x_1}
+ \sum_{j\ne 1}df^j(\bar u)\bar u_{x_1}\psi_{x_j}
 -\sum_{j\ne 1}\bar u_{x_jx_1}\psi_{x_j}
- \sum_{j}(\bar u_{x_1}\psi_{x_j} )_{x_j}.
\end{align*}
Subtracting, and using \eqref{keyderivs} and
\ba\label{keyderivs2}
\bar u_{x_j} + v_{x_j}&= \tilde u_{x_j} + \tilde u_{x_1}\psi_{x_j},\\
\bar u_t + v_t&= \tilde u_t + \tilde u_{x_1}\psi_t,\\
\ea
so that
\ba\label{solvedderivs2}
\tilde u_{x_j}-\bar u_{x_j} -v_{x_j}&=
-(\bar u_{x_1}+v_{x_1}) \frac{\psi_{x_j}}{1+\psi_{x_1}},\\
\tilde u_{t}-\bar u_{t} -v_{t}&=
-(\bar u_{x_1}+v_{x_1}) \frac{\psi_{t}}{1+\psi_{x_1}},\\
\ea
we obtain
\begin{align*}
u_t + \sum_j df^j(u){u}_{x_j} -\sum_j {u}_{x_jx_j}&=
(\partial_t-L)\bar{u}'(x_1)\psi
+v_{x_1}\psi_t - v_t \psi_{x_1}
\\ &\quad
+ \sum_{j\ne 1}(df^j(\tilde u)\tilde u_{x_1}- df^j(\bar u) \bar u_{x_1})
\psi_{x_j}
\\ &\quad
 -\sum_{j\ne 1}(\tilde u_{x_jx_1}- \bar u_{x_jx_1}) \psi_{x_j}
- \sum_{j}( (\tilde u_{x_1}- \bar u_{x_1}) \psi_{x_j} )_{x_j}.
\end{align*}

Using $v_{x_1}\psi_t - v_t \psi_{x_1} = (v\psi_t)_{x_1}-(v\psi_{x_1})_t$,
$$
df^j(\tilde u)\tilde u_{x_1}=
f(u)_{x_1}- df^j(\tilde u)\tilde u_{x_1} \psi_{x_1}
=
f(u)_{x_1}(1-\psi_x)- df^j(\tilde u)\tilde u_{x_1} \psi_{x_1}^2
,
$$
and
$ \tilde u_{x_jx_1}
= (\tilde u_{x_j})_{x_1} - \tilde u_{x_jx_1}\psi_{x_1}
= (\tilde u_{x_j})_{x_1}(1 - \psi_{x_1} )
+ \tilde u_{x_jx_1}\psi_{x_1}^2, $
and rearranging, we obtain
\begin{align*}
u_t + \sum_j df^j(u){u}_{x_j} -\sum_j {u}_{x_jx_j}&=
(\partial_t-L)\bar{u}'(x_1)\psi
+(v\psi_t)_{x_1} - (v \psi_{x_1})_t
\\ &\quad
+ \sum_{j\ne 1}(f^j(u)- f^j(\bar u))_{x_1}) \psi_{x_j}\\
&\quad
- \sum_{j\ne 1} f(u)_{x_1}\psi_{x_1} \psi_{x_j}
-\sum_{j\ne 1} df^j(\tilde u)\tilde u_{x_1} \psi_{x_1}^2
\psi_{x_j}\\
&\quad
 -\sum_{j\ne 1} (\tilde u_{x_j}- \bar u_{x_j})_{x_1} \psi_{x_j}
 +\sum_{j\ne 1} (\tilde u_{x_j})_{x_1}\psi_{x_1}\psi_{x_j}\\
&\quad
 +\sum_{j\ne 1} \tilde u_{x_jx_1}\psi_{x_1}^2\psi_{x_j}
\\ &\quad
- \sum_{j}( v_{x_1}\psi_{x_1} )_{x_j}
- \sum_{j}
\Big( (\bar u_{x_1}+v_{x_1})
\frac{\psi_{x_j}\psi_{x_1}}{1+\psi_{x_1}}\Big)_{x_j}.
\end{align*}
Noting that
$$
(f^j(u)- f^j(\bar u))_{x_1}) \psi_{x_j}=
((f^j(u)- f^j(\bar u) \psi_{x_j})_{x_1}
-(f^j(u)- f^j(\bar u)) \psi_{x_jx_1},
$$
$$
 f(u)_{x_1}\psi_{x_1} \psi_{x_j}=
 (f(u) \psi_{x_1} \psi_{x_j})_{x_1}-
 f(u) (\psi_{x_1} \psi_{x_j})_{x_1},
$$
and
$$
(\tilde u_{x_j}- \bar u_{x_j})_{x_1} \psi_{x_j}
=
((\tilde u_{x_j}- \bar u_{x_j}) \psi_{x_j} )_{x_1}
-(\tilde u_{x_j}- \bar u_{x_j}) \psi_{x_jx_1} ,
$$
with
$ |f^j(u)- f^j(\bar u)| =O(|v|)$ and
$|\tilde u_{x_j}- \bar u_{x_j}| =O(|v|), $
we obtain the result
\end{proof}

\begin{proof}[Proof of Theorem \ref{main2}]
The result of Lemma \ref{l:pert} is the only part of
the analysis that differs essentially from that of the one-dimensional
case.  The cancellation and nonlinear damping arguments go through
exactly as before to yield the analogs of Propositions
\ref{p:cancellation} and \eqref{damping}.
Likewise, we obtain a Duhamel representation analogous to
\eqref{u}--\eqref{psidot}, forming a closed system
in variables $(v,\psi_x,\psi_t)$.

To obtain the analog of Lemma \ref{sclaim}, completing the proof
of nonlinear stability,
we can carry out a somewhat simpler argument than in
the one-dimensional case, using Corollary \ref{greenbds} with $d=2$, $q=2$
for all estimates, not only the final bootstrap argument, giving
in place of \eqref{sest} the estimate
\ba\label{sestd2}
|v(\cdot,t)|_{L^p(x)}& \le
C(1+t)^{-(1-1/p)}E_0 +
C\zeta(t)^2\int_0^{t} (1+t-s)^{-(1/2-1/p)}(t-s)^{-\frac{1}{2}}
(1+s)^{-1}ds\\ &
\le
 C(E_0+\zeta(t)^2) (1+t)^{-(1-1/p)},
\ea
\ba\label{psiestd2}
|(\psi_x,\psi_t)(\cdot,t)|_{L^p(x)}& \le
C(1+t)^{-(1-1/p)-\frac{1}{2}}E_0 \\
&\quad +
C\zeta(t)^2\int_0^{t} (1+t-s)^{-(1/2-1/p)}(t-s)^{-\frac{1}{2}}
(1+s)^{-1}ds\\ &
\le
 C(E_0+\zeta(t)^2) (1+t)^{\eps -(1-1/p)-\frac{1}{2}}
\ea
for divergence-form source terms, and
\ba\label{sestd2nondiv}
|v(\cdot,t)|_{L^p(x)}& \le
C\zeta(t)^2\int_0^{t} (1+t-s)^{-(1/2-1/p)}
(1+s)^{-\frac{3}{2}}ds\\
&
\le
 C(E_0+\zeta(t)^2) (1+t)^{-(1-1/p)},
\ea
\ba\label{psiestd2n}
|(\psi_x,\psi_t)(\cdot,t)|_{L^p(x)}& \le
C(1+t)^{-(1-1/p)-\frac{1}{2}}E_0 \\
&\quad +
C\zeta(t)^2\int_0^{t} (1+t-s)^{-(1/2-1/p)}(t-s)^{-\frac{1}{2}}
(1+s)^{-\frac{3}{2}}ds\\ &
\le
 C(E_0+\zeta(t)^2) (1+t)^{\eps -(1-1/p)-\frac{1}{2}}
\ea
for faster-decaying nondivergence-form source terms.

We omit the details, which are entirely similar to, but
substantially simpler than, those of the one-dimensional case.
\end{proof}

%
%


\end{document}